\documentclass{amsart}
\usepackage{graphicx}
\newtheorem{thm}{Theorem}[section]

\newtheorem{lem}[thm]{Lemma}
\newtheorem{prop}[thm]{Proposition}

\newtheorem*{prob}{Problem}
\theoremstyle{definition}

\newtheorem{rem}[thm]{Remark}
\numberwithin{equation}{section}

\begin{document}

\title[On the geometric flows solving K\"ahlerian inverse $\sigma_k$ equations]{On the geometric flows solving K\"ahlerian inverse $\sigma_k$ equations}%

\author{Hao Fang, Mijia Lai}%
\address{Hao Fang, 14 MLH, Iowa City, IA, 52242}
\email{haofang@math.uiowa.edu}
\address{Mijia Lai, 915 Hylan Building,
University of Rocheste, RC Box 270138, Rochester, NY 14627}%
\email{lai@math.rochester.edu}%

\thanks{The first-named author is supported in part by National Science Foundation grant DMS-1008249.}

\begin{abstract}
In this note, we extend our previous work on the inverse $\sigma_k$ problem. Inverse $\sigma_{k}$ problem is a fully nonlinear geometric PDE on compact K\"ahler manifolds. Given a proper geometric condition, we prove that a large family of nonlinear geometric flows converges to the desired solution of the given PDE.
\end{abstract}
\maketitle
\section{Introduction}

In this note, we study general flows for the inverse $\sigma_k$-curvature problem in K\"ahler geometry. This is a continuation of our previous work ~\cite{FLM}.

Geometric curvature flow has been a central topic in the recent development of geometric analysis. The $\sigma_k$-curvature problems and inverse $\sigma_k$-curvature problems, fully nonlinear in nature, have appeared in several geometric settings. B.Andrews ~\cite{A1,A2} studies the curvature flow of embedded convex hypersurfaces in the Euclidean space.  Several authors study the $\sigma_k$-equation in conformal geometry, see e.g., ~\cite{V,CGY,GW,B} and references therein. It is thus interesting to explore the corresponding problem in K\"ahler geometry.

In K\"ahler geometry, special cases of the $\sigma_k$-problem have appeared in earlier literature. Among them one important example is Yau's seminal work on the complex Monge-Amp\`{e}re equations in Calabi conjecture. The general case has been studied recently in~\cite{H,HMW}. There exist, however, some analytical difficulties to completely solve this problem for $k<n$.

Another important example is the Donaldson's $J$-flow~\cite{D}, which gives rise to an inverse $\sigma_1$ type equation. $J$-flow is fully studied in~\cite{Ch1,Ch2,SW}. The general case is described and treated in ~\cite{FLM}, via a specific geometric flow. In contrast to the $\sigma_{k}$-problem, we can pose nice geometric conditions to overcome the analytical difficulties for the inverse $\sigma_{k}$-problem. In this note, we construct more general geometric flows to solve this problem.

We now describe the problem in more details.

Let $(M, \omega)$ be a compact K\"{a}hler manifold without boundary. Let $\chi$ be another K\"{a}hler metric in the class $[\chi]$ different that $[\omega]$. We define
\[
\sigma_k(\chi)={n \choose k}\frac{\chi^k\wedge \omega^{n-k}}{\omega^n},
\]
for a fixed integer $1\leq k\leq n$.
It is easy to see that $\sigma_k(\chi)$ is a global defined function on $M$; and point-wise it is the $k$-th elementary symmetric polynomial on the eigenvalues of $\chi$ with respect to $\omega$. Define
\[
c_k:=\frac{\int_M \sigma_{n-k}(\chi)}{\int_M \sigma_n(\chi)}={n \choose k}\frac{[\chi]^{n-k}\cdot [\omega]^k}{[\chi]^n},
\]
which is a topological constant depending only on cohomology classes $[\chi]$ and $[\omega]$.

In~\cite{FLM}, we studied the following problem
\begin{prob}
Let $(M,\omega)$, $\chi$ and $c_k$ be given as above, is there a metric $\tilde{\chi}\in[\chi]$ satisfying
\begin{align} \label{equation}
c_k\tilde{\chi}^n={n \choose k} \tilde{\chi}^{n-k}\wedge \omega^k ?
\end{align}
\end{prob}

To tackle this problem, we consider the following geometric flow:
\begin{align} \label{flow-old}
\left\{
\begin{array}{l l}
\frac{\partial}{\partial t} \varphi &=c_k^{1/k}-(\frac{\sigma_{n-k}(\chi_{\varphi})}{\sigma_n(\chi_{\varphi})})^{1/k}
\\
\varphi(0)& =0
\end{array}
\right.
\end{align}
in the space of K\"{a}hler potentials of $\chi$:
\[
\mathcal{P}_{\chi}:=\{ \varphi \in C^{\infty}(M) | \chi_{\varphi}:=\chi+ \frac{\sqrt{-1}}{2}\partial\bar{\partial} \varphi>0 \}.
\]

It is easy to see that the stationary point of the flow corresponds to the solution of (\ref{equation}).

When $k=1$, (\ref{flow-old}) is the Donaldson's $J$-flow~\cite{D}, defined in the setting of moment map (cf. ~\cite{Ch1}). In this case, Song and Weinkove~\cite{SW} provide a necessary and sufficient condition for the flow to converge to the critical metric. For general $k$, this problem is solved in~\cite{FLM} with an analogous condition which we  now describe.

We define $\mathcal{C}_k(\omega)$ to be
\begin{align} \label{cone}
\mathcal{C}_k(\omega)=&\{ [\chi]>0,\, \exists \chi' \in [\chi], \\ \notag
&\text{such that}\, \  nc_k \chi'^{n-1}-{n \choose k}(n-k)\chi'^{n-k-1}\wedge \omega^k>0\footnotemark\}.
\end{align}
\footnotetext[1]{a positive $(n-1,n-1)$ form.}

Note for $k=n$, (\ref{cone}) holds for any K\"{a}hler class. Hence $\mathcal{C}_n(\omega)$ is the entire K\"{a}hler cone of $M$.

The necessity of the cone condition (\ref{cone}) is easy to see once we write (\ref{equation}) locally as
\[
\frac{\sigma_{n-k}(\chi)}{\sigma_n(\chi)}=\sigma_k(\chi^{-1})=c_k.
\]
Here $\chi^{-1}$ denotes the inverse matrix of $\chi$ under local coordinates. Since $\chi^{-1}>0$, we necessarily have
\[
\sigma_k(\chi^{-1}|i)<c_k, \forall \  i.
\]
This condition is equivalent to the cone condition (\ref{cone}). See Proposition 2.4 of~\cite{FLM}.

The main result of~\cite{FLM}  is the following:
\begin{thm}\label{old}
Let $(M, \omega)$ be a compact K\"{a}hler manifolds. Let $k$ be a fixed integer $1\leq k \leq n$. Assume $\chi$ is another K\"{a}hler form with its class $[\chi] \in \mathcal{C}_k(\omega)$, then the flow
\begin{align}
\frac{\partial}{\partial t} \varphi &=c_k^{1/k}-(\frac{\sigma_{n-k}(\chi_{\varphi})}{\sigma_n(\chi_{\varphi})})^{1/k},
\end{align}
with any initial value $\chi_{_0} \in [\chi]$ has long time existence and converges to a unique smooth metric $\tilde{\chi}\in [\chi]$ satisfying
\begin{align}
c_k\tilde{\chi}^n={n \choose k} \tilde{\chi}^{n-k}\wedge \omega^k.
\end{align}
\end{thm}

In this note, we generalize Theorem~\ref{old}. We shall study an abstract flow on $M$ of the form:

\begin{align} \label{flow}
\left\{
\begin{array}{l l}
\frac{\partial}{\partial t} \varphi &=F(\chi_{\varphi})-C,
\\
\varphi(0)& =0,
\end{array}
\right.
\end{align}
where $$F(\chi_{\varphi})=f[\frac{\sigma_{n-k}(\chi_{\varphi})}{\sigma_n(\chi_{\varphi})}],$$ $$f\in C^\infty( \mathbb{R}_{> 0},\  \mathbb{R}),$$
$$C=f(c_k).$$

Note that (\ref{flow-old}) is a special case of (\ref{flow}) for $f(x)=-x^{1/k}$.

Abusing notation, we also regard $F$ as a symmetric function on
\[
\Gamma_n:=\{ \chi \in \mathbb{R}^n| \chi_1>0, \chi_2>0, \cdots \chi_n>0\}
\]
by treating $F(\chi_{\varphi})=F(\chi_1,\cdots,\chi_n)$,  where $(\chi_i)$ are eigenvalues of $\chi_{\varphi}$ with respect to $\omega$.
Then by carefully examining the proof of Theorem~\ref{old} in~\cite{FLM}, we need the following structure conditions on $F$:

\begin{itemize}
  \item Ellipticity: ${F_i}>0$,
  \item Concavity: ${ F_{i j}}\leq 0$,
  \item Strong concavity: ${ F_{ij}}+{F_i\over\chi_j}\delta_{ij}\leq 0$.
\end{itemize}
Here $F_i= {\partial F \over \partial\chi_i}$ and $F_{ij}={\partial^2 F \over{\partial\chi_i \partial\chi_j}}$. Note that concavity of $F$ follows from strong concavity and ellipticity of $F$.

It is easy to check that $F(\chi_1,\cdots, \chi_n):=-(\frac{\sigma_{n-k}(\chi)}{\sigma_n(\chi)})^{1/k}$ satisfies above conditions.

In this note, we prove the following:

\begin{thm} [Main theorem]
\label{mainthm}
Let $(M, \omega)$ be a compact K\"{a}hler manifold. $k$ is a fixed integer $1\leq k\leq n$. Let $\chi$ be another K\"{a}hler metric such that $[\chi] \in \mathcal{C}_k$. Assume that $f\in C^\infty( \mathbb{R}_{>0},\  \mathbb{R})$ satisfies the following conditions:
\begin{align}\label{condition1}
 f'<0, \quad
f''\geq0, \quad
 f''+\frac{f'}{x}\leq0,
\end{align}
then the flow (\ref{flow}) with any initial value $\chi_{_0}\in [\chi]$ has long time existence and the metric $\chi_{\varphi}$ converges in $C^{\infty}$ norm to the critical metric $\tilde{\chi}\in [\chi]$ which is the unique solution of (\ref{equation}).
\end{thm}

\begin{rem}
The novelty of our theorem is that there exists a large family of nonlinear geometric flows which yields the convergence towards the solution of inverse $\sigma_k$ problem (\ref{equation}). For example, the function $f$ can be chosen as $f(x)=-\ln x$ or $f(x)=-x^p$, for $0<p\leq1$. Note that for the special case $f(x)=-\ln x$ and $k=n$, we get an analogue of the K\"ahler-Ricci flow. For $f(x)=-x$ and $k=n$, a similar flow was studied in a recent paper~\cite{CK}.
\end{rem}

\begin{rem}
Theorem~\ref{mainthm} is inspired and can be viewed as a K\"ahler analogue of Andrew's result ~\cite{A2} on pinching estimates of evolutions of convex hypersurfaces. In fact, our structure conditions are very similar to those of his.
\end{rem}
The rest of the paper is organized as follows: In Section 2, we discuss the conditions on $f$ and strong concavity of $F$; In Section 3, we give the proof of the main result.

\section{Strong concavity}
In this section, we will explore concavity properties for functions involving the quotient of elementary symmetric polynomials.

\begin{prop}
Let $\chi\in \Gamma_n$, $f: \mathbb{R}_{>0}\to \mathbb{R}$, define $\rho(\chi_1, \cdots, \chi_n)=f(\frac{\sigma_{n-k}(\chi)}{\sigma_n(\chi)})$, suppose $f$ satisfies the following conditions:
\begin{align}  \label{condition}
 f'<0, \quad
f''\geq0, \quad
 f''+\frac{f'}{x}\leq0.
\end{align}
Then $\rho$ satisfies
\begin{itemize}
  \item Ellipticity: $\rho_i>0, \forall i$,
  \item Concavity: $\rho_{ij}\leq 0$,
  \item Strong Concavity: $\rho_{ij}+\frac{\rho_i}{\chi_j}\delta_{ij}\leq 0$.
\end{itemize}
\end{prop}

We shall refer to the conditions in (\ref{condition}) as the structure conditions on $f$.

The proof is based on the following two propositions:

\begin{prop}
Let $g(\chi_1, \cdots, \chi_n)=\log \sigma_k(\chi)$ and $\chi\in \Gamma_n$, then
\begin{itemize}
  \item $g_i>0$,
  \item $g_{ij}\leq 0$,
  \item $g_{ij}+\frac{g_i}{\chi_j}\delta_{ij}\geq0$.
\end{itemize}
\end{prop}

\begin{prop}
Let $h(\chi_1, \cdots,\chi_n):=-g(\frac{1}{\chi_1}, \cdots, \frac{1}{\chi_n})=-\log \sigma_k(\chi^{-1})$ and $\chi \in \Gamma_n$, then
\begin{itemize}
  \item $h_i>0$,
  \item $h_{ij}\leq 0$,
  \item $h_{ij}+\frac{h_i}{\chi_j}\delta_{ij}\leq 0$.
\end{itemize}
\end{prop}

We refer reader to the appendix of ~\cite{FLM} for detailed proof of Propositions 2.2 and 2.3.

\begin{proof} [Proof of Proposition 2.1]
Direct computation shows
\[
\rho_i =-f' \sigma_{k-1}(\chi^{-1}|i)\frac{1}{\chi_i^2} >0.
\]
Concavity of $\rho$ follows from strong concavity and $\rho_i>0$, hence it is suffice to show that
\[
\rho_{ij}+\frac{\rho_i}{\chi_j} \delta_{ij}\leq 0.
\]
Direct computation yields
\begin{align}
\rho_{ij}+\frac{\rho_i}{\chi_j}\delta_{ij} &=f''\sigma_{k-1}(\chi^{-1}|i)\sigma_{k-1}(\chi^{-1}|j)\frac{1}{\chi_i^2}\frac{1}{\chi_j^2}\\ \notag&+f'(\sigma_{k-2}(\chi^{-1}|i,j)\frac{1}{\chi_i^2}\frac{1}{\chi_j^2}(1-\delta_{ij})+\sigma_{k-1}(\chi^{-1}|i)\frac{1}{\chi_i^3}\delta_{ij}.
\end{align}
Since $f''+\frac{f'}{x}\leq 0$, $f''\geq0$, we have
\begin{align}
\rho_{ij}+\frac{\rho_i}{\chi_j}\delta_{ij}&\leq f''\{\frac{\sigma_{k-1}(\chi^{-1}|i)\sigma_{k-1}(\chi^{-1}|j)}{\chi_i^2\chi_j^2}\\ \notag &-\sigma_{k}(\chi^{-1})[\frac{\sigma_{k-2}(\chi^{-1}|i,j)}{\chi_i^2\chi_j^2}(1-\delta_{ij})+\frac{\sigma_{k-1}(\chi^{-1}|i)}{\chi_i^3}\delta_{ij}]\}\leq0.
\end{align}
The last inequality follows from Proposition 2.3 and the fact
\begin{align}
h_{ij}+\frac{h_i}{\chi_j}\delta_{ij}&=\frac{1}{\sigma_k(\chi^{-1})^2}\{\frac{\sigma_{k-1}(\chi^{-1}|i)\sigma_{k-1}(\chi^{-1}|j)}{\chi_i^2\chi_j^2}\\ \notag &-\sigma_{k}(\chi^{-1})[\frac{\sigma_{k-2}(\chi^{-1}|i,j)}{\chi_i^2\chi_j^2}(1-\delta_{ij})+\frac{\sigma_{k-1}(\chi^{-1}|i)}{\chi_i^3}\delta_{ij}]\}.
\end{align}
\end{proof}

For a hermitian matrix $A=(a_{i\bar{j}})$, let its eigenvalues be $\chi=(\chi_1,\cdots, \chi_n)$. For $f\in C^{\infty}(\mathbb{R}_{>0}, \mathbb{R})$, we define
$$F(A):=\rho(\chi_1,\cdots, \chi_n)= f(\frac{\sigma_{n-k}(\chi)}{\sigma_n(\chi)}).$$
Denote\[
F^{i\bar{j}}:=\frac{\partial F}{\partial a_{i\bar{j}}}, \quad F^{i\bar{j},k\bar{l}}:=\frac{\partial^2 F}{\partial a_{i\bar{j}} a_{k\bar{l}}}.
\]

It is a classical result that the properties of $F(A)$ follow from those of $\rho(\chi)$, see e.g., Theorem 1.4 of~\cite{S}. In particular, Proposition 2.1 leads to the following:

\begin{prop} Let $F(A)$ be defined as above, let $f\in C^{\infty}( \mathbb{R}_{>0}, \mathbb{R})$ satisfing (\ref{condition}). Then $F$ satisfies:
\begin{itemize}
  \item Ellipticity: $F^{i\bar{j}}>0$,
  \item Concavity: $F^{i\bar{j}, k\bar{l}}\leq 0$,
  \item Strong concavity: at $A=\text{diag}(\chi_1, \cdots, \chi_n)$, $F^{i\bar{i},j\bar{j}}+\frac{F^{i\bar{i}}}{\chi_j}\delta_{ij}\leq 0$.
\end{itemize}
\end{prop}

\section{Proof of the main theorem}

\subsection{Long time existence}\

Differentiating the flow (\ref{flow}), we get
\begin{align}\notag
\frac{\partial}{\partial t}(\frac{\partial \varphi}{\partial
t})=F^{i\bar{j}}(\chi)\partial_i\partial_{\bar{j}}(\frac{\partial \varphi}{\partial
t}).
\end{align}
From Proposition 2.4, $\frac{\partial \varphi}{\partial
t}$ satisfies a parabolic equation. By maximum principle, we have
\[
\min_{t=0} \frac{\partial
\varphi}{\partial t} \leq\frac{\partial \varphi}{\partial t}\leq \max_{t=0} \frac{\partial
\varphi}{\partial t},
\]
thus
\[
\min F(\chi_{_0}) \leq F(\chi_{\varphi})=f(\sigma_{k}(\chi_{\varphi}^{-1})) \leq \max F(\chi_{_0}).
\]
By the monotonicity of $f$, there exist two universal positive constants $\lambda_1$ and $\lambda_2$ such that
\begin{align} \label{bound}
\lambda_1 \leq \sigma_k(\chi_{\varphi}^{-1})\leq \lambda_2.
\end{align}
(\ref{bound}) implies $\chi_{\varphi}$ remains K\"{a}hler; i.e., $\chi_{\varphi}>0$. Also note that with the bound (\ref{bound}), regarding the estimate aspect, $f$, $f'$, $f''$ are all bounded.

Concerning the behavior of the flow (\ref{flow}) for an arbitrary triple data $(M, \omega, \chi)$, we have:
\begin{thm}
Let $(M, \omega, \chi)$ be given as above, the general inverse $\sigma_k$ flow (\ref{flow}) has long time existence.
\end{thm}

\begin{proof}
Following~\cite{Ch2}, we derive a time-dependent $C^2$ estimates for the potential $\varphi$. Since $\chi_{\varphi}>0$, it is suffice to derive an upper bound for $G:=tr_{\omega} \chi_{\varphi}=g^{p\bar{q}}\chi_{p\bar{q}}$.  By a straightforward computation, we get

\begin{align} \label{2.5}
\frac{\partial G}{\partial t} & = g^{p\bar{q}}F^{i\bar{j}, k\bar{l}}\chi_{i\bar{j},p}\chi_{k\bar{l},\bar{q}}+g^{p\bar{q}}F^{i\bar{j}}\chi_{i\bar{j},p\bar{q}} \\
\notag & = F^{i\bar{j}}(g^{p\bar{q}} \chi_{p\bar{q}})_{i\bar{j}} +g^{p\bar{q}}F^{i\bar{j}, k\bar{l}}\chi_{i\bar{j},p}\chi_{k\bar{l},\bar{q}} + g^{p\bar{q}} F^{i\bar{j}}( \chi_{m\bar{q}}R_{pi\bar{j}}^{m}-\chi_{m\bar{j}}R_{pi\bar{q}}^{m}).
\end{align}

The second term is non-positive by the concavity of $F$. For last term, by choosing normal coordinates, it is easy to see
\begin{align} \label{2.4}
   g^{p\bar{q}} F^{i\bar{j}}( \chi_{m\bar{q}}R_{pi\bar{j}}^{m}-\chi_{m\bar{j}}R_{pi\bar{q}}^{m}) \leq C_3+C_4 G
\end{align}
for two universal positive constants.

Now the upper bound of $G$ follows from standard maximum principle. Consequently, we have long time existence for the flow (\ref{flow}).
\end{proof}

In what follows, we will give the proof of the main theorem. Following~\cite{FLM}, we first derive a partial $C^2$ estimate for the potential $\varphi$ depending on the $C^0$ norm of $\varphi$ when the condition $[\chi]\in \mathcal{C}_k(\omega)$ holds. Then we follow the method developed in ~\cite{SW} to get uniform $C^0$ estimate and the convergence of the flow.

\subsection{Partial $C^2$ estimate}\

Without loss of generality, we can assume initial metric $\chi_{_0}$ is the metric $\chi'$ in $[\chi]$ satisfying cone condition (\ref{cone}). Since different initial data differs by a fixed potential function, the same estimates will carry over. Again, since $\chi_{\varphi}>0$, it is suffice to bound $\chi_{\varphi}$ from above. Consider $G(x,t,\xi):=\log
(\chi_{i\bar{j}}\xi^i\xi^{\bar{j}}) - A \varphi$, for $x\in M$ and $\xi\in \mathbf{T}_{x}^{(1,0)}M$ with
$g_{i\bar{j}}\xi^{i}\xi^{\bar j}=1$. $A$ is a constant to be determined. Assume $G$ attains maximum at
$(x_{_0},t_{_0})\in M\times [0,t]$, along the direction
$\xi_{_0}$. Choose normal coordinates of $\omega$ at $x_{_0}$, such
that $\xi_{_0}=\frac{\partial}{\partial z_1}$ and $(\chi_{i\bar{j}})$ is diagonal
at $x_{_0}$. By the definition of $G$, it is easy to see that
$\chi_{1\bar{1}}=\chi_{_1}$ is the largest eigenvalue of
$\{\chi_{i\bar{j}}\}$ at $x_{_0}$. We can
assume $t_{_0}>0$, otherwise we would be done. Thus locally we consider $H:=\log \chi_{1\bar{1}}-A\varphi$ instead, which achieves its maximum at $(x_{_0},t_{_0})$ as well.

For simplicity, we denote $\chi=\chi_{\varphi}$. At $x_{_0}$, assume that $\chi=\text{diag}(\chi_1, \cdots, \chi_n)$ with $\chi_1\geq \chi_2 \cdots \geq \chi_n>0$. We shall use $\chi$ to denote the hermitian matrix $(\chi_{i\bar{j}})$ or the set of the eigenvalues of $\chi_{\varphi}$ interchangeably when no confusion arises.

We compute the evolution of $H$:
\begin{align}\notag
\frac{\partial H}{\partial
t}&=\frac{\chi_{1\bar{1},t}}{\chi_{1\bar{1}}}-A\frac{\partial \varphi}{\partial
t}=\frac{F^{i\bar{j}}\chi_{i\bar{j},1\bar{1}}+F^{i\bar{j}, k\bar{l}}\chi_{i\bar{j},1}\chi_{k\bar{l},\bar{1}}}{\chi_{1\bar{1}}}-A\frac{\partial \varphi}{\partial
t},\\ \notag
H_{i\bar{i}}&=\frac{\chi_{1\bar{1},i\bar{i}}}{\chi_{1\bar{1}}}-\frac{|\chi_{1\bar{1},i}|^2}{\chi_{1\bar{1}}^2}-A\varphi_{i\bar{i}}.
\end{align}

By the maximum principle, at $(x_{_0},t_{_0})$ we have

\begin{eqnarray} \label{3.2}
0&\leq& \frac{\partial H}{\partial t} - \sum_{i=1}^{n}F^{i\bar{i}}H_{i\bar{i}}\\ \notag
 &=&
  \frac{1}{\chi_{1\bar{1}}}
 F^{i\bar{i}}(\chi_{i\bar{i},1\bar{1}}-\chi_{1\bar{1},i\bar{i}})-A\frac{\partial \varphi}{\partial t} +AF^{i\bar{i}} \varphi_{i\bar{i}}+B,
\end{eqnarray}
where $$B=\frac{1}{\chi_{1\bar{1}}}\sum_{1\leq i,j,k,l\leq
n}F^{i\bar{j},k\bar{l}}\chi_{i\bar{j},1}\chi_{k\bar{l},\bar{1}}+\sum_{i=1}^{n}F^{i\bar{i}}\frac{|\chi_{1\bar{1},i}|^2}{\chi_{1\bar{1}}^2}$$
is the collection of all terms involving the third order derivatives.

We claim that $B\leq 0$, whose proof will be presented at the end of this section. Assuming that, (\ref{3.2}) leads to
\begin{align} \label{3.3}
\frac{1}{\chi_{1\bar{1}}}
 F^{i\bar{i}}(\chi_{i\bar{i},1\bar{1}}-\chi_{1\bar{1},i\bar{i}})\geq A\frac{\partial \varphi}{\partial t} -AF^{i\bar{i}} \varphi_{i\bar{i}}.
\end{align}

We simplify the left hand side of (\ref{3.3}) by the Ricci identity:
\begin{align} \label{3.4}
LHS
& =
\frac{1}{\chi_{1\bar{1}}}\sum_{i=1}^{n}F^{i\bar{i}}(\chi_{i\bar{i}}R_{i\bar{i}1\bar{1}}-\chi_{1\bar{1}}R_{1\bar{1}i\bar{i}})
\\ \notag
&\leq  \frac{C_1\sum_{i=1}^n F^{i\bar{i}}\chi_{i}}{\chi_{1\bar{1}}} -\sum_{i=1}^{n} F^{i\bar{i}}
R_{1\bar{1}i\bar{i}}\\\notag &\leq
\frac{C_0}{\chi_{1\bar{1}}}+C_2 \sum_{i=1}^{n} F^{i\bar{i}}.
\end{align}
Note that for the bound on $\sum_{i=1}^n F^{i\bar{i}}\chi_{i}$ we have used  (\ref{bound}) and the following computation:
\begin{align}
\sum_{i=1}^n F^{i\bar{i}}\chi_{i}&= -f' \sum_{i=1}^n \sigma_{k-1}(\chi^{-1}|i) \frac{1}{\chi_i^2} \chi_i\\
\notag &= -f' \sum_{i=1}^n \sigma_{k-1}(\chi^{-1}|i) \frac{1}{\chi_i}\\
\notag &= -kf' \sigma_k(\chi^{-1}) \leq C.
\end{align}

To deal with the right hand side of (\ref{3.3}), we divide into two cases:

{\bf Case 1: $k< n$}\

In this case, we have the following technical lemma due to the cone condition.
\begin{lem} \label{lemma1}
For $k<n$, assume that $\chi_{_0}=\chi'\in [\chi]$ is a K\"{a}hler form satisfying the cone condition (\ref{cone}), also assume that $C_1\leq \sigma_k(\chi^{-1}) \leq C_2$ for two universal constants, then there exists a universal constant $N$, such that if $\frac{\chi_1}{\chi_n}\geq N$, then there exists a universal constant $\theta>0$, such that
\begin{align}
\sigma_k^{\frac{1}{k}}(\frac{\chi_{0i\bar{i}}}{\chi_i^2})\geq (1+\theta) c_k^{-1/k}\sigma_k^{2/k}(\chi^{-1}).
\end{align}
\end{lem}
We refer reader to Theorem 2.8 of ~\cite{FLM} for a proof.\\

{\bf Case 1.(a): $\frac{\chi_1}{\chi_n}\geq N$}, where $N$ is given in Lemma~\ref{lemma1}\

Applying the Lemma~\ref{lemma1}, we claim that
there exists a universal constant $\epsilon>0$ such that
\begin{align}  \label{3.7}
\frac{\partial \varphi}{\partial t} -F^{i\bar{i}}\chi_{i\bar{i}} + (1-\epsilon) F^{i\bar{i}}\chi_{0i\bar{i}}\geq 0.
\end{align}

Indeed, by direct computation, we have
\begin{align}
\sum_{i=1}^n F^{i\bar{i}}\chi_{0i\bar{i}}&=-f' \sum_{i=1}^n \sigma_{k-1}(\chi^{-1}|i) \frac{\chi_{0i\bar{i}}}{\chi_i^2}\\
\notag &\geq-kf' \sigma_{k}^{1-1/k}(\chi^{-1})\sigma_{k}^{\frac{1}{k}}(\frac{\chi_{0i\bar{i}}}{\chi_i^2})\\
  \notag &\geq-kf' \sigma_{k}^{1-1/k}(\chi^{-1}) (1+\theta)c_k^{-1/k} \sigma_k^{2/k} (\chi^{-1}).
\end{align}
The first inequality follows from the G{\aa}rding's inequality.

Therefore by taking $\epsilon$ such that $(1-\epsilon)(1+\theta)=1$, (\ref{3.7}) is reduced to
\begin{align} \label{3.9}
\frac{\partial \varphi}{\partial t} -F^{i\bar{i}}\chi_{i\bar{i}}-kf'\sigma_k^{1+1/k}(\chi^{-1}) c_k^{-1/k} \geq 0.
\end{align}

By scaling, we can assume $c_k=1$, and modifying $f$ by adding a constant, we can further assume that $f(1)=0$. Plugging $F^{i\bar{i}}$ and letting $x=\sigma_k(\chi^{-1})$, (\ref{3.9}) is equivalent to
\begin{align}
f(x)+kf'(x) x- kf'(x) x^{1+1/k}\geq 0.
\end{align}
The inequality above holds provided $f''+\frac{f'}{x}\leq 0$ and $f(1)=0$.

Combining (\ref{3.3}), (\ref{3.4}) and (\ref{3.7}), we have
\begin{align}  \label{3.11}
 A\epsilon\sum_{i=1}^n F^{i\bar{i}}\chi_{0i\bar{i}} \leq \frac{C_1}{\chi_1}+C_2\sum_{i=1}^n F^{i\bar{i}}.
\end{align}

Since $\chi_{_0}$ is a fixed form, there exists a universal constant $\lambda>0$ such that
\[
A\lambda \sum_{i=1}^n F^{i\bar{i}} \leq A\epsilon\sum_{i=1}^n F^{i\bar{i}}\chi_{0i\bar{i}}.
\]

Hence in (\ref{3.11}), taking $A$ such that $A\lambda-C_2=1$, an upper bound for $\chi_1$ will follow once we have shown $\sum_{i=1}^n F^{i\bar{i}}$ is bounded from below. For that we have
\begin{align}
\sum_{i=1}^n F^{i\bar{i}}&=-f' \sum \sigma_{k-1}(\chi^{-1}|i) \frac{1}{\chi_i^2}\\ \notag
 &\geq -kf' \sigma_k^{1-1/k}(\chi^{-1}) \sigma_k^{\frac{1}{k}}(\frac{1}{\chi_i^2}) \geq \tilde{C} \sigma_k^{1+1/k}(\chi^{-1}) \geq C.
\end{align}

{\bf Case 1.(b): $\frac{\chi_1}{\chi_n}\leq N$}\

In this case, the upper bound for $\chi_1$ follows directly from the lower bound (\ref{bound}) on $\sigma_k(\chi^{-1})$. Since
\begin{align}
\lambda_1\leq \sigma_k(\chi^{-1})\leq { n\choose k} \frac{1}{\chi_n^k},
\end{align}
then we get an upper bound for $\chi_n$, thus an upper bound for $\chi_1$ as $\chi_1\leq N\chi_n$.\\

{\bf Case 2: $k=n$}\

In this case, we proceed (\ref{3.3}) directly. Since we only concern $f$ on the closed interval $[\lambda_1, \lambda_2]$, we can assume that $f$ is positive by adding a constant.
By (\ref{3.4}), we have that
\begin{align} \label{3.18}
\text{LHS of (\ref{3.3})}\leq \frac{C_0}{\chi_1} +C_2 \sum_{i=1}^n F^{i\bar{i}}  \leq C_3 \sum_{i=1}^{n} \frac{1}{\chi_i}.
\end{align}

For right hand side, we have
\begin{align} \label{3.19}
\text{RHS of (\ref{3.3})}\geq A(-f(c_k)+nf'\sigma_n(\chi^{-1}))+A\epsilon C_4 \sum_{i=1}^n \frac{1}{\chi_i}.
\end{align}

Combining (\ref{3.18}) and (\ref{3.19}), taking $A$ such that $A\epsilon C_4-C_3=1$, we find there exists a universal constant $C$, such that
\begin{align}
\sum_{i=1}^n \frac{1}{\chi_i} \leq C.
\end{align}
Consequently, we have lower bound on $\chi_i, \forall i$, thus upper bound for $\chi_1$ by (\ref{bound}).

Thus we have proved that there exists a universal constant $C$ such that

$$\chi_{1}\leq C.$$ This leads to

\begin{thm} \label{C2}Let notations are given as above, we have
\begin{align}\notag
|\partial \bar{\partial}\varphi|_{C^0}\leq Ce^{A\varphi-\inf_{M\times[0,t]}\varphi},
\end{align}
for two universal constants $A$ and $C$ for any time interval $[0,t]$.
\end{thm}

Finally, we prove the claim that
\[
B=\frac{1}{\chi_{1\bar{1}}}\sum_{1\leq i,j,k,l\leq
n}F^{i\bar{j},k\bar{l}}\chi_{i\bar{j},1}\chi_{k\bar{l},\bar{1}}+\sum_{i=1}^{n}F^{i\bar{i}}\frac{|\chi_{1\bar{1},i}|^2}{\chi_{1\bar{1}}^2}\leq 0.
\]
We will divide $B$ into three groups:

\[X=\frac{1}{\chi_{1\bar{1}}}\sum_{1\leq i,j\leq n }F^{i\bar{i},j\bar{j}}\chi_{i\bar{i},1}\chi_{j\bar{j},\bar{1}} +F^{1\bar{1}}\frac{|\chi_{1\bar{1},1}|^2}{\chi_{1\bar{1}}^2}.\] The fact that $X$ is non-positive follows from the strong concavity of $F$ in Proposition 2.4.

\[
Y=\frac{1}{\chi_{1\bar{1}}}\sum_{i=2}^n F^{i\bar{1}, 1\bar{i}} \chi_{i\bar{1},1}\chi_{1\bar{i},\bar{1}}+\sum_{i=2}^{n}F^{i\bar{i}}\frac{|\chi_{1\bar{1},i}|^2}{\chi_{1\bar{1}}^2}.
\]
One can show by direct computation that $F^{i\bar{1},1\bar{i}}+\frac{F^{i\bar{i}}}{\chi_1} \leq 0, \forall i$, thus $Y\leq 0$.

\[
Z=\frac{1}{\chi_{1\bar{1}}}\sum_{i\neq j, j>1, k\neq l, k>1}F^{i\bar{j},k\bar{l}}\chi_{i\bar{j},1}\chi_{k\bar{l},\bar{1}}.
\]
Again by direct computation, each term is non-positive. We have thus finished the proof of the claim.

\subsection{$C^0$ estimate and convergence of the flow}
Following the method in ~\cite{SW}, we introduce two functionals. The monotonic behavior of these functionals along the flow (\ref{flow}) yields the $C^0$ estimate and convergence of the flow. Define functionals in $\mathcal{P}_{\chi_0}$ by
\begin{align} \label{functional1}
\mathcal{F}_{k,\chi_{_0}}(\phi)=\mathcal{F}_k(\phi)=\int_0^1 \int_M  \dot{\phi_t}
\chi_{\phi_t}^{k}\wedge\omega^{n-k} dt,
\end{align}
where $\phi_t$ is an arbitrary smooth path
in $\mathcal{P}_{\chi_{_0}}$ connecting $0$ and $\phi$, and $\dot{\phi_t}$ denotes time derivative. One can readily check that this definition is independent of the choice of the path $\varphi_t$.
Moreover, define
\begin{align} \label{functional2}
\mathcal{F}_{k,n}(\phi)={n \choose k}\mathcal{F}_k(\phi)-c_{n-k}\mathcal{F}_n(\phi).
\end{align}

The first variation of $\mathcal{F}_{n-k,n}$ is
\begin{align} \notag
\frac{d}{ d t}\mathcal{F}_{n-k,n}(\phi)=\int_M \dot{\phi_t}({n \choose k}\chi_{\phi_t}^{n-k}\wedge \omega^k-c_k \chi_{\phi_t}^n).
\end{align}
It follows the Euler-Lagrange equation of $\mathcal{F}_{n-k,n}$ is precisely the critical equation (\ref{equation}):
\begin{align} \notag
c_k \chi_{\phi}^n={n \choose k}\chi_{\phi}^{n-k}\wedge \omega^k.
\end{align}

We quote the following uniqueness result, the Theorem 4.1 of ~\cite{FLM}.
\begin{prop}[Uniqueness]
The solution to the critical equation (\ref{equation}) is unique up to a constant.
\end{prop}

\begin{prop} [Monotonicity of $\mathcal{F}_{n-k,n}$]
The functional $\mathcal{F}_{n-k,n}$ is decreasing along the flow (\ref{flow}).
\end{prop}

\begin{proof}
By direct computation, we have
\begin{eqnarray}
\frac{d}{d t}\mathcal{F}_{n-k,n}(\varphi_t)&=& \int_M \dot{\varphi_t}({n \choose k}\chi_{\varphi}^{n-k}\wedge \omega^k-c_k\chi_{\varphi}^n)\\
\notag &=& \int_M (f(\sigma_k(\chi_{\varphi}^{-1}))-f(c_k))(\sigma_k(\chi_{\varphi}^{-1})-c_k) \chi_{\varphi}^n < 0.
\end{eqnarray}
The integrand is of the form $(f(a)-f(b))(a-b)$ which is negative since $f'<0$.
\end{proof}

\begin{prop} [Monotonicity of $\mathcal{F}_{n-k}$]
The functional $\mathcal{F}_{n-k}$ is non-increasing along the flow (\ref{flow}).
\end{prop}
\begin{proof}
First define $g(x)=f(\frac{1}{x})$. It follows that $g$ is concave iff $f''+\frac{f'}{x}\leq 0$. Then by Jensen's inequality, we have
\begin{align}
\frac{1}{\int_M \chi^{n-k}\wedge \omega^k}\int_M f(\sigma_k(\chi^{-1})) \chi^{n-k}\wedge \omega^k&= \frac{1}{\int_M \chi^{n-k}\wedge \omega^k}\int_M g(\frac{\sigma_n(\chi)}{\sigma_{n-k}(\chi)}) \chi^{n-k}\wedge \omega^k \\ \notag
& \leq
 g(\frac{1}{\int_M \chi^{n-k}\wedge \omega^k} \int_M \frac{\sigma_n(\chi)}{\sigma_{n-k}(\chi)}\chi^{n-k}\wedge \omega^k)\\
 \notag &=g (\frac{1}{c_k})=f(c_k).
\end{align}
Hence
\begin{align}
\frac{\partial }{\partial t}\mathcal{F}_{n-k} = \int_M (f(\sigma_k(\chi_{\varphi}^{-1}))-f(c_k))\chi_{\varphi}^{n-k}\wedge \omega^k\leq 0.
\end{align}
\end{proof}

Finally, we single out the essential steps for the rest of the proof. By Theorem 4.5 of~\cite{FLM}, we have uniform bounds for the oscillation of $\varphi_t$, i.e., $$||\sup \varphi_t-\inf \varphi_t||\leq C.$$ Then using functional $\mathcal{F}_{n-k}$, we obtain a suitable normalization $\hat{\varphi_t}$ of $\varphi_t$, for which we can get uniform $C^0$ estimates, thus uniform $C^2$ estimates by Theorem~\ref{C2}. Higher order estimates follow from the Evans-Krylov and Schauder estimates. The corresponding metric thus converges to the critical metric solving inverse $\sigma_k$ problem (\ref{equation}).

\end{document}